\documentclass[11pt]{amsart}

\usepackage{epigamath}

\usepackage[english]{babel}

\numberwithin{equation}{section}

\usepackage[shortlabels]{enumitem}
\setlist[enumerate,1]{label={\rm(\arabic*)}, ref={\rm\arabic*}} 

\usepackage{amsfonts, amsthm, amssymb, amsmath, stmaryrd}
\usepackage{mathrsfs,array}
\usepackage{tikz}
\usepackage{tikz-cd}
\usepackage{rotating}
\usetikzlibrary{arrows,shapes,positioning}
\usetikzlibrary{decorations.markings}
\tikzstyle arrowstyle=[scale=1]
\usepackage{pgfplots}
\pgfplotsset{compat=newest}
\usepackage{graphicx}
\usepackage{rotating}
\usepackage{cleveref}
\usepackage{mathtools}

\newtheorem{thm}{Theorem}[section]
\newtheorem{lem}[thm]{Lemma}
\newtheorem{prop}[thm]{Proposition}

\theoremstyle{definition}
\newtheorem{defn}[thm]{Definition}

\theoremstyle{remark}
\newtheorem{rem}[thm]{Remark}
\newtheorem{exa}[thm]{Example}

\newtheorem{remque}[thm]{Remark/Question}

\DeclareMathOperator{\supp}{supp}

\DeclareMathOperator{\sepgon}{sepgon}
\DeclareMathOperator{\Sep}{Sep}

\usepackage{url}

\EpigaVolumeYear{10}{2026} \EpigaArticleNr{1} \ReceivedOn{July 4, 2024}
\InFinalFormOn{August 2, 2025}
\AcceptedOn{August 22, 2025}

\title{Real plane separating $(M-2)$-curves of degree $d$ and totally real pencils of degree $d-3$}
\titlemark{Real plane separating $(M-2)$-curves of degree $d$ and totally real pencils of degree $d-3$}

\author{Matilde Manzaroli}
\address{Mathematisches Institut der Universit\"at T\"ubingen, Auf der Morgenstelle 10, 72076 T\"ubingen, Germany}
\email{matilde.manzaroli@uni-tuebingen.de}

\authormark{M.~Manzaroli}

\AbstractInEnglish{A non-singular real plane projective curve of degree $5$ with five connected components is separating if and only if its ovals are in non-convex position. In this article, this property is set into a different context and generalised to all real plane separating $(M-2)$-curves.}



\acknowledgement{The author thanks the Athene Grant and the DFG, German Research Foundation (Deutsche Forschungsgemeinschaft), Project- ID 286237555, TRR 195, ``Symbolic Tools in Mathematics and their Application''.}

\begin{document}



\maketitle

\begin{prelims}

\DisplayAbstractInEnglish

\bigskip

\DisplayKeyWords

\medskip

\DisplayMSCclass

\end{prelims}


\newpage

\setcounter{tocdepth}{1}

\tableofcontents


\section{Introduction}
Let $C$ be any smooth complex compact algebraic curve equipped with an anti-holomorphic involution $\sigma:C \rightarrow C$, \textit{i.e.}~a smooth real algebraic compact curve. If the real points $C(\mathbb R)$ of $C$ separate its complex points $C (\mathbb C)$, \textit{i.e.}~$C (\mathbb C) \setminus C(\mathbb R)$ is disconnected, we say that $C$ is of \textit{type I} or \textit{separating}. 

By Harnack--Klein's inequality, see \cite{Harn76, Klei73}, the number $l$ of connected components of $C(\mathbb R)$ is bounded by the genus $g$ of $C$ plus~$1$. For any fixed $0 \leq i \leq g+1$, if $l$ equals $g+1-i$, we say that $C$ is an $(M-i)$-curve. The number $l$ is related to the  separateness property of the curve. For example, if $C$ is separating, then $l$ has parity $g+1$. Or, if $C$ is an $M$-curve, then $C$ is separating. In this article, we focus on non-singular real algebraic plane projective separating $(M-2)$-curves; see Theorem~\ref{thm: unico}.

First of all, let us present some general features of separating curves. If $C$ is of type I, the two  halves of $C (\mathbb C)\setminus C(\mathbb R )$ induce two opposite orientations on $C(\mathbb R )$ called \textit{complex orientations}; see \cite{Rokh74}. Looking at complex orientations of separating real curves embedded in some ambient surface has lead to remarkable progress in the study of their topology and a refinement of their classifications. One of the first results relating topology, complex orientations and properties of separating plane curves is Rokhlin's complex orientations formula, see \cite{Rokh74,Mish75}, and one of the more recent is \cite[Theorem 1.1]{Orev21}, where Orevkov shows that there are finer relations for the numbers which intervene in the complex orientations formula. An important role in \cite{Orev21} is played by separating morphisms.

\begin{defn}
\label{defn: separating}
We say that a real morphism $f$ from a smooth real algebraic compact curve $C$ to the complex projective line $\mathbb P^1_{\mathbb C} $ is \textit{separating} if $f^{-1}(\mathbb P^1 (\mathbb R ) )= C (\mathbb R )$. 
\end{defn}

According to Ahlfors \cite[Section~4.2]{Ahlf50}, there exists a separating morphism $f: C \rightarrow \mathbb P^1_{\mathbb C}$ if and only if~$C$ is of type I. We call \textit{separating gonality} of $C$, and we denote by $\sepgon(C)$, the minimal possible value for the degree of a separating morphism of $C$. Observe that the separating gonality always has the number of real connected components of $C(\mathbb R)$ as lower bound.

Actually, there is a more general definition of separating morphisms which includes real morphisms between any real algebraic varieties of the same dimension. We direct the interested reader to \cite{KumSha20} and \cite{KumLeTMan22}. In the current paper, we need Definition~\ref{defn: separating} only.

The study of smooth real curves of type I and their separating morphisms has been carried out mainly from two points of view: on the one hand, from that of abstract curves and, on the other hand, from that of curves embedded in some ambient surface; see \textit{e.g.}~\cite{Huis01,Gaba06,CopHui13,Copp13,Copp14,KumShaw20,Orev21}. 

Let $C$ be a real separating curve of genus $g$. If $C$ is an $M$-curve, then it admits a separating morphism of degree $g+1$, because of the Riemann--Roch theorem;  for details see \cite[Proof of Theorem 1.7]{KumShaw20}.
In general, for some fixed integers $i,k$ such that $1 \leq i \leq g+1$ and $k \geq g+1-i$, if $C$ is an $(M-i)$-curve, it is not evident, \textit{a priori}, if there exists and how to construct a separating morphism $f: C \rightarrow \mathbb P^1_{\mathbb C}$ of degree $k$. 

A remarkable result, due to Gabard, see \cite[Theorem 7.1]{Gaba06}, states that a genus $g$ real separating curve~$C$ with $l$ real connected components admits a separating morphism of degree at most $\frac{g+ l+1}{2}$. If $C$ is an $(M-2)$-curve, Gabard's result tells us that $\sepgon(C)$ equals $g$ or $g-1$.

From now on, unless otherwise stated, we call \textit{real plane $($separating\,$)$ curve} any real algebraic plane projective (separating) curve.

The real locus of a non-singular real plane curve is homeomorphic to a disjoint union of circles embedded in $\mathbb P^2 (\mathbb{R})$. One can embed a circle in $\mathbb P^2(\mathbb{R})$ in two different ways: as an \textit{oval}, \textit{i.e.}~realising the trivial class in $H_1( \mathbb P^2 (\mathbb{R}); \mathbb{Z}/2\mathbb{Z})$, or as a \textit{pseudo-line}. A non-singular real plane curve of even degree has only ovals (possibly none); if the curve is of odd degree, it has exactly one pseudo-line and ovals (possibly none). 

The real projective plane is separated by an oval in two disjoint non-homeomorphic connected components: a disk, called the \textit{interior} of the oval, and  a M\"obius band, called the \textit{exterior} of the oval. 

Even if the following lemma is a well-known result in the study of the topology of real plane curves, we could not find \Cref{lem: quintic} in the literature. So, we include its statement and proof here. 

The proof of \Cref{lem: quintic} is revealing for two reasons:
\begin{itemize}
\item On the one hand, it contains the germ of the main theorem of this article, \Cref{thm: unico}, which is a generalisation of \Cref{lem: quintic} to higher-degree real plane separating curves.
\item On the other hand, its skeleton would be quite hard to follow to prove most of its generalisations  to higher-degree curves; therefore, the proof of \Cref{lem: quintic} marks the distance between the real plane quintic case and \Cref{thm: unico}.
\end{itemize}

\begin{lem}
	\label{lem: quintic}
A non-singular real plane curve $C_5$ of degree $5$ with five connected components is separating if and only if its ovals are in non-convex position $($see Definition~\ref{defn: non-convex}\,$)$.
\end{lem}

\begin{figure}[h!]
\begin{picture}(100,105)
\put(-25,-3){\includegraphics[width=0.25\textwidth]{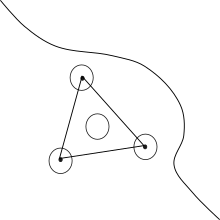}}
\end{picture}
\caption{Arrangement of a triplet $(\mathbb P^2 (\mathbb R), C(\mathbb R), S_1 \cup S_2 \cup S_3)$ as in Definition~\ref{defn: non-convex}, where the $S_i$ are the three segments.}
\label{fig: non_convex}
\end{figure} 

The interested reader can find one of the possible constructions of a real plane separating quintic with five connected components in \cite[Section~2, Figure~19, p.~36]{Viro07}. 

\begin{defn}
	\label{defn: non-convex}
	Let $C$ be a non-singular real plane curve of degree $5$ with five connected components. We say that its ovals are in non-convex position if three of the ovals of $C(\mathbb R)$ are such that, having chosen a point in the interior of each of them, once one traces three segments $S_1, S_2, S_3$ containing the points pairwise and such that every segment does not cross the pseudo-line of $C(\mathbb R)$, the fourth oval is contained inside the triangle cut out by these segments; see Figure~\ref{fig: quintic}.
\end{defn}
In order to prove \Cref{lem: quintic}, we first recall Definitions~\ref{defn: totally_real_pencil} and~\ref{defn: pos_neg_oval}. Then, we state Lemmas~\ref{lem: Fied_what_we_need} and~\ref{lem: no_nest_cx_orientations}, which are, respectively, a reformulation of \cite[Lemma 3]{Fied83} and a restriction to a special case of the complex orientations formula of \cite{Mish75}. We present Fiedler and Mishachev's results in this form for the reader's convenience and in order to avoid introducing more notation than necessary to understand the rest of the paper.

\begin{defn}
 	\label{defn: totally_real_pencil}
 	Let $C$ be a non-singular real algebraic plane projective curve of type I. We say that $C$ admits a totally real pencil of curves of degree $k$ if there exists an integer $k$ such that there are $f,g \in \mathbb R [x,y,z]_{k}$ such that $V(\lambda f + \mu g) \cap C$ consists of only real points for all $\lambda, \mu \in \mathbb R$ that are not both zero.
 \end{defn}
 
 \begin{rem}
 \label{rem: from_pencil_to_sep_morphism}
 	 Let us consider a non-singular real algebraic curve $C$ admitting a real pencil of degree $k$ as in Definition~\ref{defn: totally_real_pencil}. Set $B$ as the base locus of the pencil on the curve $C$, \textit{i.e.}~the intersection points of $V( f)$, $ V(g)$ and $C$. Note that $B$ consists of a finite number of points. Then, one can define a separating morphism from~$C$ to the projective line as follows.
 	 First, one defines a degree $k- |B|$ morphism $\tilde{h}:C \setminus B \rightarrow \mathbb P^1_{\mathbb C}$ sending a point $x$ to $[\lambda: \mu]$, where $\lambda f + \mu g$ is the defining polynomial of the curve of the pencil passing through $x$. Finally, such a morphism $\tilde{h}$ can always be uniquely extended to a degree $k -|B|$ morphism $h: C \rightarrow \mathbb P^1_{\mathbb C},$ in this case a separating morphism.
 	 \end{rem}

\begin{defn}
	\label{defn: pos_neg_oval}\leavevmode
	\begin{itemize}
	\item Let $\mathcal O$ and $J$ be, respectively, an oval and the pseudo-line of an odd-degree non-singular real plane separating curve equipped with one of its complex orientations. The orientation of the curve induces orientations on $\mathcal O$ and on $J$. Then, the oval is called positive if $[\mathcal O]= -2[J]$ in $H_1(N; \mathbb{Z})$, where $N$ is the closure of the non-orientable component of $\mathbb P^2(\mathbb R) \setminus \mathcal O$; otherwise, $\mathcal O$ is called negative. Note that an oval is either positive or negative independently of the chosen complex orientation.
	\item The total number of positive (resp.\ negative) ovals of an odd-degree non-singular real plane separating curve is denoted by $\Lambda_+$ (resp.\ $\Lambda_-$).
\end{itemize}
\end{defn}

\begin{lem}[\textit{cf.} \protect{\cite[Lemma 3]{Fied83}}]
\label{lem: Fied_what_we_need}
Let $C$ be a non-singular real plane separating curve of odd degree $d=2k+1$. 
If there exist two real distinct lines $L_0$ and $L_1$ such that
\begin{enumerate}
	\item each $L_i$ intersects $C$ in $d-1$ real points of which exactly one point $p_i$ has multiplicity two $($a tangency point$)$ and the others have multiplicity one; 
	\item the points $p_i$, for $i=1,2$, belong to two distinct ovals $ \mathcal O_1$ and $\mathcal O_2$;
	\item the intersection point of\, $L_0$ and $L_1$ is not contained in $C$;
\end{enumerate}
and if there exist two real numbers $a,b$ such that all the lines $\{L_t\}_{t \in [a,b]}$ of the pencil of lines with base point $s$ are such that $L_a = L_0$,  $L_b = L_1$ and each $L_t$ intersects $C(\mathbb R)$ transversally in exactly $d-2$ real points, for all $t \in (a,b)$; then one of the ovals $\mathcal O_1$ and $\mathcal O_2$ is positive, and the other is negative $($see Definition~\ref{defn: pos_neg_oval}\,$)$. 
\end{lem}

\begin{lem}[Special case of the complex orientations formula of \cite{Mish75}]
\label{lem: no_nest_cx_orientations}	
Let $C$ be a non-singular real plane separating curve of odd degree $d=2k+1$ and with $l$ real connected components. Assume that all ovals of\, $C$ are empty, \textit{i.e.}~the interior of each oval does not contain any other oval of\, $C$. Then, the following formula holds: 
\begin{equation}
	\Lambda_+ - \Lambda_- = l-1-k(k+1);
\end{equation}
	see the notation in Definition~\ref{defn: pos_neg_oval}.
\end{lem}

\begin{figure}[h!]
\begin{picture}(100,129)
\put(-175,0){\includegraphics[width=1\textwidth]{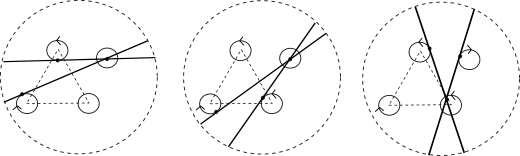}}
\end{picture}
\caption{$(\mathbb P^2 (\mathbb R) \setminus J, C_5(\mathbb R), L_0(\mathbb R) \cup L_1(\mathbb R))$. The arrows denote the fixed complex orientation of the ovals of $C(\mathbb R)$, and the dots $\bullet$ denote the points of tangency and the intersection point of the fixed lines $L_0$ and $L_1$, which are in bold.}
\label{fig: fiedler_ex}
\end{figure} 

\begin{proof}[Proof of \Cref{lem: quintic}]
 Assume that $C_5$ is separating and, for the sake of contradiction, suppose that the ovals are not in non-convex position, in the sense of Definition~\ref{defn: non-convex}. On one hand, applying \Cref{lem: Fied_what_we_need}, one can show that two ovals must be negative and two positive (see \Cref{defn: pos_neg_oval}, where the pseudo-line of $C_5$ is depicted at infinity and the dashed triangle is just a reminder of the ovals' position assumption), as follows. Pick two real lines $L_0, L_1$ intersecting in the interior of one of the ovals of $C_5$ as in the picture on the left of Figure~\ref{fig: fiedler_ex}; then, one can apply \Cref{lem: Fied_what_we_need}, finding two ovals, one positive and one negative; afterwards, one can fix a complex orientation of $C_5$. Therefore, in order to find out the complex orientation of the remaining two ovals, one iterates and uses \Cref{lem: Fied_what_we_need} two more times; see the last two pictures in Figure~\ref{fig: fiedler_ex}, reading it from left to right.
 
 On the other hand, the complex orientations formula in \Cref{lem: no_nest_cx_orientations} implies that three ovals must be negative and one positive. We get a contradiction. Therefore, the ovals must be in non-convex position.
 
 Now, assume that the ovals are in non-convex position. In order to prove that $C_5$ is separating, it is enough to prove that $C_5$ admits a totally real pencil; see \Cref{defn: totally_real_pencil} and \Cref{rem: from_pencil_to_sep_morphism}. Let us use the notation of Definition~\ref{defn: non-convex} and set $\Delta$ as the union of the segments $S_1,S_2,S_3$. Consider any real pencil $\mathcal{P}$ of conics with base locus the vertices $p_1, p_2, p_3$ of $\Delta$ plus another point $p$ contained in the oval not intersecting~$\Delta$. By construction any real conic of the pencil intersects each oval of $C_5 (\mathbb R)$ in two points. Therefore, in order to show that $\mathcal P$ is totally real and end the proof, it is enough to show that any real conic of $\mathcal P$ also intersects the pseudo-line of $C_5 (\mathbb R)$.

 Now, take a point $q$ in the connected component of $\mathbb{P}^2 (\mathbb R) \setminus \Delta$ not containing $p$. Let us consider the projection $\pi: \mathbb{P}^2 (\mathbb R) \setminus \{q\} \rightarrow \mathbb R^2$. Since any real conic $C$ of the pencil $\mathcal{P}$ is such that $\pi (C (\mathbb R))$ is convex, any side $S_i$ of $\Delta$, deprived of its extremities, must be either contained in $\pi (C (\mathbb R))$ or not. Moreover, remark that 
 \begin{itemize}
 	\item for the line $L_i$ containing $S_i$, either $S_i$ is contained in $C (\mathbb R)$, or $L_i \setminus S_i$ is contained in $C (\mathbb R)$;
 	\item if at least one among $S_1,S_2,S_3$ is not contained in $\pi (C (\mathbb R))$, then $C$ must intersect the pseudo-line. 
 \end{itemize}
 On the other hand, it is not possible that all sides of $\Delta$ are contained in $\pi (C (\mathbb R))$, by the convexity of the affine conic and the fact that all the points $p_1,p_2,p_3,p$ belong to $\pi (C (\mathbb R))$. It follows that all real conics of~$\mathcal P$ intersect all real connected components of $C_5$ for a total of ten real intersection points  for each conic, and the pencil is indeed totally real with respect to $C_5$.
\end{proof} 

A priori, Lemma~\ref{lem: quintic} is uniquely an observation concerning separating plane quintics with four ovals. It is not clear that one may expect to have some generalisation of it to other real plane separating curves. On the other hand, a remarkable fact, from the proof of \Cref{lem: quintic}, is that \textit{any non-singular real plane projective separating curve of degree $5$ with five connected components admits a totally real pencil of conics.} Surprisingly, this property can be generalised to all real plane separating $(M-2)$-curves; see \Cref{thm: unico}. Indeed, the generalisation of \Cref{lem: quintic} to separating plane $(M-2)$-curves $C_d$ of degree $d$ is that every curve $C_d$ admits totally real pencils of curves of degree $d-3$. 

Since the separating gonality of a smooth real $(M-2)$-curve of genus $g$ and of type I is either $g-1$ or $g$, the main result of this paper is as follows.

\begin{thm}
	\label{thm: unico}
	Let $C$ be a non-singular real plane separating $(M-2)$-curve of degree $d \geq 4$ and of separating gonality $g-1$ or $g$, where $g=\frac{(d-1)(d-2)}{2}$ denotes the genus of\, $C$. Then, the curve $C$ admits infinitely many totally real pencils of degree $d-3$ with $g-1$, respectively $g-2$, base points on $C$.
	\end{thm}

In \Cref{thm: unico}, the difference in the number of possible base points of the totally real pencils admitted by the curve $C$ only depends on the fixed separating gonality of $C$; see \Cref{rem: from_pencil_to_sep_morphism}. Also, the proof of \Cref{thm: unico} is slightly different with respect to the fixed separating gonality; therefore, in Section~\ref{sec: proofs}, we split it into the proofs of, respectively, Propositions~\ref{prop: sepgon_m-2_g-1} and~\ref{prop: sepgon_m-2_g}.

 \begin{rem}
 	\label{rem: tot_real_pencil_no_base_points}
 	In \cite{Touz13}, there have been constructed totally real pencils of rational cubics for real plane separating $(M-2)$-sextics realising two particular topological arrangements in the real projective plane. On the other hand, the points of the base locus of such totally real pencils do not belong to the curves; therefore, the obtained separating morphisms of such sextics have degree 18.
 
 	\end{rem}

\subsection*{Acknowledgements}
I would like to thank Erwan Brugall\'e and the referees.

\section{Proofs and examples}
\label{sec: proofs}	
	In order to prove Theorem~\ref{thm: unico}, it is enough to prove Propositions~\ref{prop: sepgon_m-2_g-1} and~\ref{prop: sepgon_m-2_g}, which each deal with one of the two admissible separating gonalities. First, we state \cite[Theorem 3.2]{Orev21}, restricted to the case of real plane curves, which is used to prove both propositions.  

\begin{thm}[\textit{cf.} \protect{\cite[Theorem 3.2]{Orev21}}]
\label{thm: stepan}
Let $C$ be a non-singular real plane separating curve. Let $D$ be a real divisor belonging to the linear system $|C+K_{\mathbb P^2_{\mathbb C}}|$. Assume that $D$ does not have $C$ as a component. We may always write $D=2D_0 + D_1$ with $D_1$ a reduced curve and $D_0$ an effective divisor. Let us fix a complex orientation on $ C(\mathbb R)$ and an orientation on $\mathbb P^2(\mathbb R) \setminus ( C(\mathbb R) \cup D_1(\mathbb R))$ which changes each time we cross $ C(\mathbb R) \cup D_1(\mathbb R)$ at its smooth points. The latter orientation induces a boundary orientation on $C(\mathbb R) \setminus ( C(\mathbb R) \cap D_1)$. Let $f: C \rightarrow \mathbb P^1_{\mathbb C}$ be a separating morphism. Then it is impossible that, for some $p \in  \mathbb P^1(\mathbb R)$, the set $f^{-1}(p) \setminus \supp(D)$ is non-empty and the two orientations coincide at each point of the set.
\end{thm}

\begin{prop}
\label{prop: sepgon_m-2_g-1}
Let $C$ be a non-singular real plane $(M-2)$-curve of degree $d \geq 4$ and of type I. Assume that $\sepgon(C)=g-1$, where $g=\frac{(d-1)(d-2)}{2}$ is the genus of\, $C$. Then, the curve $C$ admits infinitely many totally real pencils of degree $d-3$ with $g-1$ base points on $C$.
\end{prop}

\begin{proof}
An important role in the following proof is played by B\'ezout's theorem; therefore, it is useful to note that an oval in the real projective plane always intersects any other real component in an even number of points (which may be zero).
	
	Since $\sepgon(C)=g-1$, there exists a separating morphism $f: C \rightarrow \mathbb P^1_{\mathbb C}$ of degree $g-1$. Therefore, for any fixed $p \in \mathbb P^1(\mathbb R)$, every point $p_i$ in $f^{-1}(p)$ belongs to a distinct connected component $C_i$ of $C(\mathbb R)$, where $1 \leq i \leq g-1$. 
	
In the following, firstly, we show that every real curve of any real pencil of curves of degree $d-3$ passing through $g-2$ points of $f^{-1}(p)$ contains all $g-1$ points of $f^{-1}(p)$. In particular, this would imply that any such pencil must be totally real for the curve $C$.

	For $\frac{(d-1)(d-2)}{2}-1 = g-1$ fixed real points, there always exists at least one real curve of degree $d-3$ passing through such a configuration, because the projective space of plane curves of degree $d-3$ has dimension $\frac{d(d-3)}{2}= g-1$. Note that a configuration of $g-2$ points chosen among $f^{-1}(p)$ is not necessarily generic. But there always exists at least one real pencil of curves of degree $d-3$ passing through such a configuration, and this fact is enough for the following.
	
For any fixed $p \in \mathbb P^1(\mathbb R)$, pick any configuration $\mathcal P$ of $g-2$ distinct points $p_1,\dots, p_{g-2}$ belonging to $f^{-1}(p)$. Applying the notation in Theorem~\ref{thm: stepan}, take $D_1$ as some degree $d-3$ curve containing $\mathcal P$ and an additional real point $q$, different from $p_{g-1}$. Note that $D_1$ must be reduced; otherwise, $D_1$ can be written as $2A+B$, where $A$ and $B$ are two real curves and have degree respectively $s$ and $d-3-2s$, with $1 \leq s \leq \lfloor \frac{d-3}{2} \rfloor$. But, this leads to a contradiction as follows. Because of the choice of the points $p_1,\dots, p_{g-2}$, the curve $C$ must intersect $A \cup B$ in at least $2(g-3)+1$ points, because $C$ has at least $g-2$ ovals and, therefore, $A \cup B$ intersects at least $g-3$  such ovals plus another point $p_i$, for some $1 \leq i \leq g-2$. B\'ezout's theorem implies that $(d-3-s)d$, the number of intersection points between $A \cup B$ and $C$, must be greater than or equal to $2(g-3)+1$. Hence $ sd \leq 3$, which is not possible, because $s$ is at least~$1$ and $d$ is greater than or equal to~$4$.

Let us fix an orientation on $\mathbb P^2(\mathbb R) \setminus ( C(\mathbb R) \cup D_1(\mathbb R))$ which changes each time we cross $ C(\mathbb R) \cup D_1(\mathbb R)$ at its smooth points. This orientation induces a boundary orientation $\mathfrak O$ on $C(\mathbb R) \setminus ( C(\mathbb R) \cap D_1)$.

Suppose, for the sake of contradiction, that $p_{g-1}$ does not belong to $D_1$. Then, the set $f^{-1}(p) \setminus \supp(D_1)= \{p_{g-1}\}$ is non-empty, and, up to fixing one of the two complex orientations on $ C(\mathbb R)$, such a complex orientation and the orientation $\mathfrak O$ coincide at $p_{g-1}$. But this contradicts \Cref{thm: stepan}.

 Therefore, any real curve of degree $d-3$ passing through the points of $\mathcal P$ must also contain  $p_{g-1}$. In particular, the configuration $\mathcal P$ defines a totally real pencil of curves of degree $d-3$. It may be that one has more than one such pencils; in that case, we just pick one.
 
Moreover, such a totally real pencil has exactly $g-1$ base points on $C$. Indeed $\mathcal B_p \cap C=f^{-1}(p)$, where $\mathcal B_p$ denotes the base locus of the pencil. The fact that $\mathcal B_p \cap C$ contains $f^{-1}(p)$ comes from the construction.
In order to show that the equality holds, 
it is enough to note that the degree of the separating morphism of~$C$ defined by the pencil (see Remark~\ref{rem: from_pencil_to_sep_morphism}) equals $d(d-3)$ minus the cardinality of $B_p$ and, therefore, there cannot be other base points in addition to $p_1, \dots, p_{g-1}$, or this would imply that the degree should be lower than $g-1=\sepgon(C)$, which contradicts the hypotheses.
\end{proof}

\begin{prop}
\label{prop: sepgon_m-2_g}
Let $C$ be a non-singular real plane $(M-2)$-curve of degree $d \geq 4$ and of type I. Assume that $\sepgon(C)=g$, where $g=g(C)=\frac{(d-1)(d-2)}{2}$ is the genus of\, $C$. Then, the curve $C$ admits infinitely many totally real pencils of degree $d-3$ with $g-2$ base points on $C$.
\end{prop}

\begin{proof}
Since $\sepgon(C)=g$, there exists a separating morphism $f: C \rightarrow \mathbb P^1$ of degree $g$. Therefore, for any fixed $p \in \mathbb P^1(\mathbb R)$, the points $p_1, \dots, p_{g-2}$ of $f^{-1}(p)$ belong to distinct connected components $C_i$ of $C(\mathbb R)$, where $1 \leq i \leq g-2$, and the remaining points $p_{g-1},p_g$ belong to the same connected component $C_{g-1}$.
In order to construct a totally real pencil for $C$, firstly, we are going to prove that any real plane curve of degree $d-3$ passing through $p_1, \dots , p_{g-2}$ intersects all real connected components of $C$ and, in particular, passes through at least two points of $C_{g-1}$. Note that the latter would imply that any such curve intersects $C$ in real points only. 

Let us apply the notation in Theorem~\ref{thm: stepan} and take $D_1$ as some degree $d-3$ real curve containing~$\mathcal P$ and an additional real point $q$. The choice may be not unique; in that case, we just pick one such real curve~$D_1$. Via a similar argument to the one used in the proof of Proposition~\ref{prop: sepgon_m-2_g-1}, one can show that $D_1$ must be reduced. 
Note that either both points $p_g, p_{g-1}$ belong to $ C_{g-1} \cap D_1$, or two other distinct points $r_1, r_2$ must belong to $C_{g-1} \cap D_1$ such that $p_g, p_{g-1}$ are contained in different connected components of $C_{g-1} \setminus \{ r_1,r_2 \}$. Indeed, let us fix an orientation on $\mathbb P^2(\mathbb R) \setminus ( C(\mathbb R) \cup D_1(\mathbb R))$ which changes each time we cross $ C(\mathbb R) \cup D_1(\mathbb R)$ at its smooth points. This orientation induces a boundary orientation $\mathfrak O$ on $C(\mathbb R) \setminus ( C(\mathbb R) \cap D_1)$. Suppose firstly, for the sake of contradiction, that $D_1$ intersects $p_g$ (resp.\ $p_{g-1}$) and does not intersect $p_{g-1}$ (resp.~$p_{g}$). Then, analogously to the argument used in the proof of Proposition~\ref{prop: sepgon_m-2_g-1}, one finds a contradiction with Theorem~\ref{thm: stepan}. Suppose, secondly, for the sake of contradiction, that $D_1$ does not intersect $C_{g-1}$ (resp.\ intersects $C_{g-1}$ in only one point $z$ different from $p_g, p_{g-1}$). Then the set $f^{-1}(p) \setminus \supp(D_1)= \{p_{g-1}, p_g\}$ is non-empty, and, up to fixing one of the two complex orientations on $ C(\mathbb R)$, such a complex orientation and the orientation~$\mathfrak O$ coincide at $p_{g-1}$ and $p_g$, because the points are contained in the same connected component $C_{g-1}$ (resp.\ $C_{g-1} \setminus \{ z \}$). But, once again, this contradicts \Cref{thm: stepan}. It follows that $D_1$ must intersect $C_{g-1}$ at least twice in order to separate the points $p_g$ and $p_{g-1}$; \textit{i.e.}~these points must belong to two distinct connected components of $C_{g-1} \setminus D_1$ so that $\mathfrak O$ and the fixed complex orientation of $C(\mathbb R)$ do not agree at both points. Hence $D_1$ intersects all real connected components of $C(\mathbb R)$.
Therefore, any real pencil of curves of degree $d-3$ passing through $p_1, \dots , p_{g-2}$ is totally real. Finally, let us prove that $B_p$, the base locus of the pencil on~$C$, equals $f^{-1}(p) \setminus \{ p_g , p_{g-1} \}$, \textit{i.e.}~the pencil has exactly $g-2$ base points on $C$.
By construction $p_1,\dots , p_{g-2}$ are contained in $B_p$. In order to show that the equality holds, it is enough to note that the degree of the separating morphism of $C$ defined by the pencil (see Remark~\ref{rem: from_pencil_to_sep_morphism}) equals $d(d-3)$ minus the cardinality of $B_p$ and, therefore, there cannot be other base points in addition to $p_1, \dots, p_{g-2}$, or this would imply that the degree should be lower than $g=\sepgon(C)$, which contradicts the hypotheses.
\end{proof}

Let us consider a non-singular real plane separating curve $C_5$ of degree $5$ with five connected components. As proved in \cite[Example 2.2]{Manz24}, applying \Cref{thm: stepan}, one can show that $C_5$ cannot have separating gonality equal to $5$. Therefore, $\sepgon(C_5)=6$. Moreover, applying \Cref{thm: stepan} once again, we observe the following.

\begin{exa}
\label{exa: quintic}
First, recall that $C_5$ must have three negative and one positive ovals (see \Cref{lem: no_nest_cx_orientations}). Let us prove that all separating morphisms of degree $6$ of $C_5$ must have odd degree on the three negative ovals and degree~$2$ either on the pseudo-line or on the positive oval; see Definition~\ref{defn: pos_neg_oval}.
\begin{figure}[h!]
\begin{picture}(100,109)
\put(-25,-10){\includegraphics[width=0.7\textwidth]{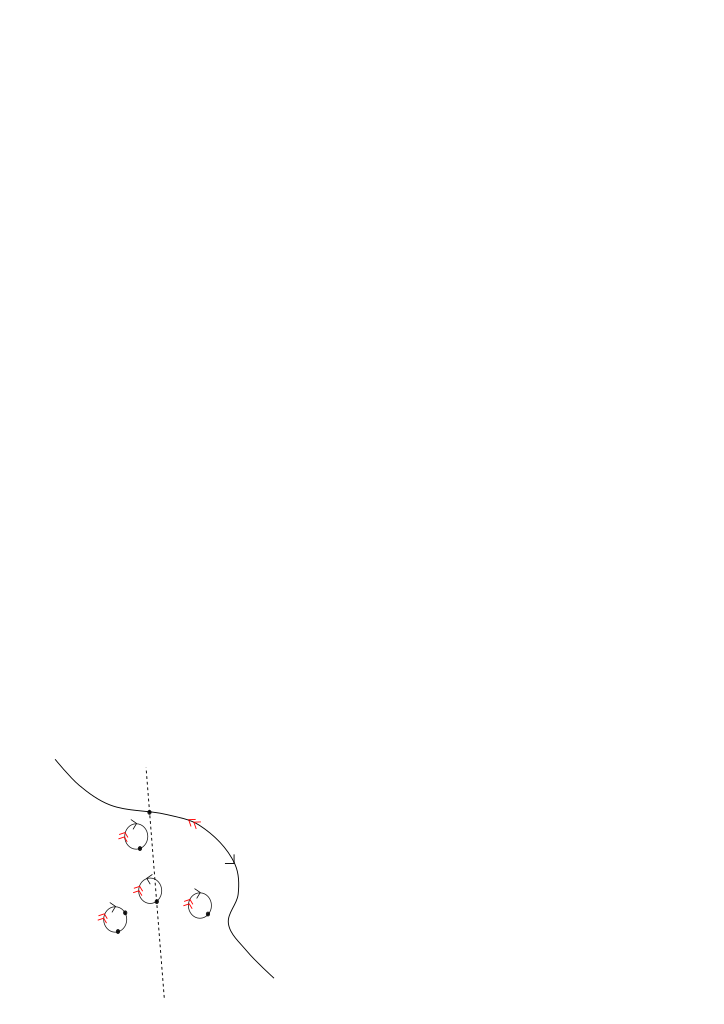}}
\end{picture}
\caption{$(\mathbb P^2 (\mathbb R), C(\mathbb R), L(\mathbb R))$ of Example~\ref{exa: quintic}. Double arrows denote $\mathfrak O$, simple arrows the fixed complex orientation of $C(\mathbb R)$ and $\bullet$ the points in $f^{-1}(p)$.}
\label{fig: quintic}
\end{figure} 
For the sake of contradiction, let us suppose that there exists a separating morphism $f: C \rightarrow \mathbb P^1_{\mathbb C}$ of degree $6$ such that $f$ has degree $2$ when restricted to a negative oval of $C(\mathbb R)$. Then, fix some $p \in \mathbb P^1 (\mathbb R)$ and apply \Cref{thm: stepan}, taking as $D_0$ the line passing through	the two points of $f^{-1}(p)$ belonging to the positive oval and the pseudo-line. Up to a choice of the orientation $\mathfrak O$ (double arrows in Figure~\ref{fig: quintic}), we get a contradiction with Theorem~\ref{thm: stepan}; therefore, such an $f$ cannot exist. On the other hand, $\sepgon(C)=6$. This means that all separating morphisms of $C$ must have odd degree on the three negative ovals.	 
\end{exa}

\begin{remque}
\label{remque: sharpness}
Let $C$ be a non-singular real plane separating curve of odd degree $d$  with $l$ real connected components. The inequalities of \cite[Theorem 1.1]{Orev21} give an upper bound for the sum of refinements of the quantities $\Lambda_+$ and $\Lambda_-$ (see \Cref{defn: pos_neg_oval}) of $C$ with respect to a number depending on $d$ and $l$ only. To be more precise, we need to introduce some notation. We say that an oval is \textit{even} (resp.\ \textit{odd}) if it is encircled by an even (resp.\ odd) number of other ovals. By convention the number of even (resp.\ odd) ovals is denoted by $p$ (resp.\ $n$). Using the notation in \cite{Orev21}, let 
\begin{center}
$\Lambda_+^p$ be the number of positive even ovals, 

$\Lambda_-^p$ be the number of negative even ovals,

$\Lambda_+^n$ be the number of positive odd ovals, 

$\Lambda_-^n$ be the number of negative odd ovals.

\end{center}

Theorem~1.1 of \cite{Orev21}, restricted to the case of real separating plane $(M-2)$-curves of genus $g$ and odd degree $2k+1$, with $k > 0$, states that the following inequalities hold:
\begin{equation}
\label{eqn: orevkov_ineq}
	\Lambda_+^p + \Lambda_-^n + 1 \geq \frac{k^2+k}{2}-1, \quad  \Lambda_+^n + \Lambda_-^p \geq \frac{k^2+k}{2}-1.
\end{equation}
For real plane quintics as in Example~\ref{exa: quintic}, one has $\Lambda_{\pm}^n=0$, $\Lambda_-^p=3$, $\Lambda_+^p=1$; therefore, the bound on the left of (\ref{eqn: orevkov_ineq}) is sharp. More in general, any odd-degree real separating plane $(M-2)$-curve of genus $g$ for which one of the two bounds in (\ref{eqn: orevkov_ineq}) is sharp must have separating gonality equal to $g$ (in fact, one can apply the same argument as in \Cref{exa: quintic}). From this observation, two questions directly follow: 
\begin{itemize}
\item For any integer $k \geq 2$, is there a non-singular real plane separating $(M-2)$-curves of degree $4k+1$ for which the bound on the left of (\ref{eqn: orevkov_ineq}) is sharp?

Observe that, because of \cite[Remark 1.8]{Orev21}, for all $k \geq 1$, such bounds can never be sharp for real separating plane $(M-2)$-curves of degree $4k+3$. Moreover, for any real plane separating curve of odd degree, the bound on the right of (\ref{eqn: orevkov_ineq}) can never be sharp.
\end{itemize}
\begin{itemize}
\item Do there exist two real plane $(M-2)$-curves of degree $d$ which have the same arrangement in the real projective plane, up to homeomorphism of $\mathbb P^2 (\mathbb R)$, but with different separating gonality?
	\end{itemize}
\end{remque}

In \cite{KumShaw20}, the separating morphisms of real separating curves are studied as follows. Let a smooth real separating algebraic compact curve $C$ consist of $l$ real connected components $C_1, \dots, C_l$. Let $f: C \rightarrow \mathbb P^1_{\mathbb C}$ be any separating morphism of $C$. Let $d_i(f) \in \mathbb N$ be the degree of the covering map $f|_{C_i}: C_i \rightarrow \mathbb P^1 (\mathbb R ) $, and set $d(f):=(d_1(f), \dots, d_l(f))$. The set $\Sep(C)$ of all such degree partitions forms a semigroup, called the \textit{separating semigroup}.

Here, we report a remark on the separating semigroup of real separating $(M-2)$-curves.

\begin{lem}
	\label{lem: sep_semigroup_m-2}
Let $C$ be a smooth real compact separating $(M-2)$-curve of genus $g \geq 2$. Then $\Sep(C) \supseteq (4,3,\dots,3)+\mathbb N^{g-1}$. Moreover, 
\begin{enumerate}
	\item $\Sep(C) \supseteq (3,\dots,3)+\mathbb N^{g-1}$ if\, $\sepgon(C)=g-1$; 
	\item $\Sep(C) \supseteq (4,2,\dots,2)+\mathbb N^{g-1}$ if\, $\sepgon(C)=g$.
\end{enumerate}
\end{lem}

\begin{proof}
In the proof, the following criterion is used to show that certain divisors $D$ on $C$ are non-special, \textit{i.e.}~the index of speciality $\dim (\mathcal{L}(K-D))$ of $D$ is zero, where $K$ is the canonical divisor on $C$. If $D$ has positive degree and the degree of $K-D= 2g-2 - $deg$(D)$ is negative, the index of speciality of $D$ is zero; see, for example, \cite[Remark IV.1.3.2 ]{Hart77}. 
There exists a real divisor $\tilde D$ on $C$ associated to a separating morphism $f: C \rightarrow \mathbb P^1$ with $\deg(f)=\sepgon(C)$, which is either $g$ or $g-1$; see \cite[Theorem~7.1]{Gaba06}. If $\deg(f)=g-1$, this means that $(1,\dots, 1) \in \Sep(C)$. Moreover, since $\Sep(C)$ is a semigroup, see \cite[Proposition 2.1]{KumShaw20}, there exists a separating morphism $\tilde f$ of $C$ of degree $3g-3$ with partition degree $(3, \dots, 3)$. Therefore, since $g \geq 2$, the non-speciality criterion implies that the associated divisor $\tilde D$ is non-special, and, by \cite[Proposition 3.2 and Remark 3.3]{KumShaw20}, the separating semigroup of $C$ contains $(3, \dots, 3) + \mathbb N^{g-1}$.

Otherwise, if $\deg(f)=g$, we have that $(2,1,\dots, 1) \in \Sep(C)$ and, analogously, there exists a separating morphism $\tilde f$ of $C$ of degree $2g$ with partition degree $(4,2, \dots, 2)$. Therefore, thanks to the same argument as in the previous paragraph, the associated divisor $\tilde D$ is non-special, and the separating semigroup of $C$ contains $(4, 2, \dots, 2) + \mathbb N^{g-1}$.
\end{proof}

\begin{rem}
  In \cite[Example 2.8]{KumShaw20} it is observed, via an example, that the separating semigroup of real separating curves is not symmetric in general. Here we give another example. Pick a real plane separating quintic $C$ with five connected components. A linear system of rank $2$ on a curve of genus greater than or equal to $3$ is unique; see \cite[A.18]{ACGH85}. So, one can label the pseudo-line, the positive oval and the negative ovals of $C(\mathbb R)$, respectively, as $X_1,\dots,X_5$. The element $(2,1,1,1,1)$ has been constructed in proof of \Cref{lem: quintic}. But, because of \Cref{exa: quintic}, not all permutations of $(2,1,1,1,1)$ belong to $\Sep(C)$. In fact, only $(1,2,1,1,1)$ may also exist.	
	\end{rem}

\begin{rem}
	\label{rem: other_surfaces}
	The interested reader can investigate, analogously to the case of plane curves, further constructions and applications of Theorem~\ref{thm: stepan} to real curves embedded in other ambient surfaces; see \textit{e.g.}~\cite{DegKha00,GudShu80,Mikh98,Manz21,Manz22,Orev03}.
\end{rem}


\providecommand{\bysame}{\leavevmode\hbox to3em{\hrulefill}\thinspace}

\end{document}